\documentclass[10pt]{amsart}

\RequirePackage{amsmath}
\RequirePackage{amsthm}
\RequirePackage{amsfonts}
\RequirePackage{amssymb}
\RequirePackage{graphicx}
\RequirePackage{hyperref}
\RequirePackage{algorithm}
\RequirePackage{algorithmic}
\RequirePackage{dcpic}
\RequirePackage[all,dvips,arc,curve,color,frame]{xy} 

\usepackage{srcltx}
\usepackage{pictexwd}
\usepackage{xcolor}

\newenvironment{thma}[1][Theorem]{\begin{trivlist}\item[\hskip \labelsep {\bfseries #1}]}{\end{trivlist}}

\newcommand{\LINEIF}[2]{%
    \STATE\algorithmicif\ {#1}\ \algorithmicthen\ {#2} %
}

{}
{}

\DeclareMathOperator{\diam}{{\bf diam}}

\DeclareMathOperator{\ncl}{{\bf ncl}}

\newcommand{\rb}[1]{{\left( #1 \right)}}

\newcommand{\gp}[1]{{\left\langle #1 \right\rangle}}
\newcommand{\gpr}[2]{{\left\langle #1 \mid #2 \right\rangle}}

\def\CP{{\mathcal P}}

\def\MN{{\mathbb{N}}}
\def\MZ{{\mathbb{Z}}}

\def\MR{{\mathbb{R}}}

\DeclareMathOperator{\code}{{code}}
\DeclareMathOperator{\Cay}{{\bf{}Cay}}
\DeclareMathOperator{\Sch}{{\bf{}Sch}}
\DeclareMathOperator{\rwr}{{\bf{}wr}}
\DeclareMathOperator{\ord}{{\bf{}ord}}

\def\tO{{\tilde O}}

\def\CP{{\mathbf{CP}}}
\def\PP{{\mathbf{PP}}}
\def\WP{{\mathbf{WP}}}

\newtheorem{theorem}{Theorem}[section]
\newtheorem{lemma}[theorem]{Lemma}

\newtheorem{definition}[theorem]{Definition}
\newtheorem{example}[theorem]{Example}

\newtheorem{remark}[theorem]{Remark}
\newtheorem{proposition}[theorem]{Proposition}
\newtheorem*{theorem*}{Theorem}
\newtheorem*{proposition*}{Proposition}

\title[Algorithmic theory of free solvable groups]{Algorithmic theory of free solvable groups:
randomized computations}

\author[A. Ushakov]{Alexander Ushakov}
\address{Stevens Institute of Technology, Hoboken, NJ, 07030 USA}
\email{aushakov@stevens.edu}

\thanks{The author was partially supported by NSF grant DMS-0914773 and by NSA Mathematical Sciences Program grant number
H98230-14-1-0128.\\The author would like to thank Andrey Nikolaev for his helpful and insightful comments.}

\begin{document}

\maketitle

\begin{abstract}
We design new deterministic and randomized algorithms for
computational problems in free solvable groups.
In particular, we prove that
the word problem and the power problem can be solved
in quasi-linear time and the conjugacy problem
can be solved in quasi-quartic time by Monte Carlo type algorithms.


\end{abstract}

\section{Introduction}

The study of algorithmic problems in free solvable groups can be traced to the work \cite{Magnus:1939}
of Magnus, who in 1939 introduced an embedding (now called the {\em Magnus embedding}) of an arbitrary
group of the type $F/[N,N]$ into a matrix group of a particular type with coefficients
in the group ring of $F/N$ (see Section \ref{se:free_solvable_gps} below).  Since
the word problem in free abelian groups is decidable in polynomial time, by induction, this embedding
gives a polynomial time decision algorithm for a fixed free solvable group $S_{r,d}$.
However the degree of the polynomial here grows together with $d$.
An algorithm polynomial in both: the length of a given word and the class $d$ of the free solvable
group was found later in \cite{Miasnikov_Romankov_Ushakov_Vershik:2010}.
It was proved that the word problem has time complexity
$O(r |w| \log_2 |w|)$ in the free metabelian group $S_{r,2}$,
and $O(r d|w|^3)$ in a free solvable group $S_{r,d}$  for $d  \geq 3$.

The general approach to the conjugacy problem in wreath products
was suggested by Matthews in \cite{Matthews:1966} who also described the solution
to the conjugacy problem in free metabelian groups.
The first solution to the conjugacy problem in free solvable groups was given
by Remeslennikov and Sokolov in \cite{Remeslennikov-Sokolov:1970} who
proved that the conjugacy in $S_{r,d}$ can be reduced to the conjugacy
in a wreath product of $S_{r,d-1}$ and a free abelian group.
Later Vassileva showed in \cite{Vassilieva:2011} that the power problem in free solvable groups
can be solved in $O(rd (|u|+|v|)^6)$ time and used that result to
show that the Matthews-Remeslennikov-Sokolov
approach can be transformed into a polynomial time $O(rd(|u|+|v|)^8)$ algorithm.
In this paper we improve the results of \cite{Miasnikov_Romankov_Ushakov_Vershik:2010}
and \cite{Vassilieva:2011}, namely we prove that:

\begin{thma}{\bf \ref{th:deterministic_WP}.}
There exists a quasi-quadratic time $\tO(|w|^2)$ deterministic algorithm
solving the word problem in $S_{r,d}$.
\end{thma}

\begin{thma}{\bf \ref{th:deterministic_PP}.}
There exists a quasi-quadratic time $\tO((|u|+|v|)^2)$ deterministic algorithm
solving the power problem in $S_{r,d}$.
\end{thma}

\begin{thma}{\bf \ref{th:deterministic_CP}.}
There exists a quasi-quintic time $\tO((|u|+|v|)^5)$ deterministic algorithm
solving the conjugacy problem in $S_{r,d}$.
\end{thma}

We can improve these results further if we grant our machine an access to a random number generator.
The price of that improvement is an occasional incorrectness of the result.
Fortunately, we can control the probability of an error:
for any fixed polynomial $p$ we can adjust some internal parameter in the algorithm to guarantee
that the probability of an error converges to $0$ as fast as $O(1/p(n))$.

\begin{thma}{\bf \ref{th:randomized_WP}.}
There exists a quasi-linear time $\tO(|w|)$ false-biased
randomized algorithm solving the word problem in $S_{r,d}$.
\qed
\end{thma}

\begin{thma}{\bf \ref{th:randomized_PP}.}
There exists a quasi-linear time $\tO(|u|+|v|)$ unbiased randomized algorithm
solving the power problem in $S_{r,d}$.
\qed
\end{thma}

\begin{thma}{\bf \ref{th:randomized_CP}.}
There exists a quasi-quartic time $\tO((|u|+|v|)^4)$ unbiased randomized algorithm
solving the conjugacy problem in $S_{r,d}$.
\qed
\end{thma}

Also, we want to mention Theorem \ref{th:geometric_conjugacy} which gives a geometric approach
to the conjugacy problem in free solvable groups.

\begin{thma}{\bf \ref{th:geometric_conjugacy}.}
Words $x,y \in F(X_r)$ represent conjugate elements in $S_{r,d}$
if and only if there exists $z\in F(X_r)$ such that
$zxz^{-1}$ and $y$ define the same flows in the Schreier graph of $\gp{y}$ in $S_{r,d-1}$.
\qed
\end{thma}

\subsection{Randomized algorithms}

A randomized algorithm is an algorithm which uses randomness
as a part of its logic. Typically it uses uniformly random bits
as an auxiliary input to guide its behavior in the hope of achieving
good performance in the average case over all possible choices of random bits.

Historically, the first randomized algorithm was a method developed by
M. Rabin in \cite{Rabin:1976} for the closest pair problem in computational geometry.
The study of randomized algorithms was spurred by the 1977 discovery
of a randomized primality test by R. Solovay and V. Strassen
in \cite{Solovay-Strassen:1977}. Soon afterwards M. Rabin in \cite{Rabin:1980}
demonstrated that the Miller's primality test can be
turned into a very efficient $\tO(\log^2 (n))$ randomized algorithm.
At that time, no practical deterministic algorithm
for primality was known.
Even though a deterministic polynomial-time $\tO(\log^6(n))$ primality
test has since been found (see AKS primality test, \cite{Agrawal-Kayal-Saxena:2004}),
it has not replaced the older probabilistic tests in cryptographic
software nor is it expected to do so for the foreseeable future.
See \cite{Motwani-Raghavan:1995} for more on randomized algorithms.
There are two main types of randomized algorithms:
Las Vegas and Monte Carlo algorithms.

A {\em Monte Carlo algorithm} is a randomized algorithm whose running time is deterministic,
but whose output may be incorrect with a certain (typically small) probability.
For decision problems, these algorithms are generally classified as either false-biased or true-biased. A false-biased Monte Carlo algorithm is always correct when it returns false; a true-biased behaves likewise.
While this describes algorithms with one-sided errors, others might have no bias; these are said to have two-sided errors. The answer they provide (either true or false) will be incorrect,
or correct, with some bounded probability.
The Solovay-Strassen primality test always answers true for prime number inputs;
for composite inputs, it answers false with probability at least $1/2$ and
true with probability at most $1/2$. Thus, false answers from the algorithm are
certain to be correct, whereas the true answers remain uncertain; this is said
to be a $(1/2)$-correct false-biased algorithm.

A {\em Las Vegas algorithm} is a randomized algorithm that always gives correct results;
that is, it always produces the correct result or it informs about the failure.
Las Vegas algorithms were introduced by L\'aszl\'o Babai in 1979, in the context
of the graph isomorphism problem, as a stronger version of Monte Carlo algorithms,
see \cite{Babai:1979}.

\subsection{Algorithmic problems in groups}

Let $F = F_r =  F(X)$ be a free group with a basis   $X=X_r = \{x_1, \ldots, x_r\}$.
By $|w|$ we denote the {\em length} of $w\in F$.
By $\varepsilon$ we denote the empty word.
When $|uv|=|u|+|v|$, then we write $u\circ v$ for $uv$.
Let $R \subseteq F$.
A pair $(X,R)$ defines a {\em presentation} of a group
$G = F / N$ (also denoted by $\gpr{X}{R}$),
where $N = \ncl(R)$ is the {\em normal closure} of $R$ in $F$.
If $R$ is finite [recursively enumerable], then the presentation is called finite [recursively enumerable].
For a recursively presented group $G$ one can study the following algorithmic questions.

\medskip
\noindent{\bf The word problem ($\WP$) in $G=\gpr{X}{R}$:}
Given $w \in F(X)$ decide if $w=1$ in $G$, or not.

\medskip
\noindent{\bf The power problem ($\PP$) in $G=\gpr{X}{R}$:}
Given $u,v \in F(X)$ compute $k\in\MZ$ such that $u=v^k$ in $\gpr{X}{R}$.
If such $k$ does not exist, then return $Fail$.

\medskip
\noindent{\bf The conjugacy problem ($\CP$) in $G=\gpr{X}{R}$:}
Given $u,v \in F(X)$ decide if there exists $c\in F(X)$
satisfying $u=c^{-1}vc$, or not.

\medskip
It is easy to see that decidability/complexity of problems above
does not depend on the generating set $X$.
See \cite{Magnus-Karrass-Solitar:1977,Lyndon-Schupp:2001} for more on
algorithmic problems in groups.

\subsection{$X$-digraphs}
\label{se:flows}

An $X$-\emph{labeled directed graph} $\Gamma$ (or an
$X$-\emph{digraph}) is a pair of sets $(V,E)$ where
the set $V$ is called the {\em vertex set} and the set
$E \subseteq V \times V \times X$ is called the {\em edge set}.
An element $e = (v_1,v_2,x) \in E$ designates an edge with
the {\em origin} $v_1$ (also denoted by $\alpha(e)$), the {\em terminus} $v_2$
(also denoted by $\omega(e)$), labeled by $x$ (also denoted by $\mu(e)$).
We often use notation $v_1\stackrel{x}{\rightarrow} v_2$
to denote the edge $(v_1,v_2,x)$. A {\em path} in $\Gamma$
is a sequence of edges $p=e_1,\ldots,e_k$ satisfying
$\omega(e_i)=\alpha(e_{i+1})$ for every $i=1,\ldots,k-1$.
The {\em origin} $\alpha(p)$ of $p$ is the vertex $\alpha(e_1)$,
the {\em terminus} $\omega(p)$ is the vertex $\omega(e_k)$,
and the {\em label} $\mu(p)$ of $p$ is the word $\mu(e_1),\ldots, \mu(e_k)$.
We say that an $X$-digraph $\Gamma$ is:
\begin{itemize}
    \item
{\em rooted} if it has a special vertex, called the root;
    \item
{\em folded} (or deterministic) if for every
$v \in V$ and $x\in X$ there exists at most one edge with the origin
$v$ labeled with $x$;
    \item
{\em complete} if for every $v_1\in V$ and $x\in X$
there exists an edge $v_1\stackrel{x}{\rightarrow} v_2$;
    \item
{\em inverse} if with every edge
$e=g_1\stackrel{x}{\rightarrow} g_2$ $\Gamma$ also contains the {\em inverse edge}
$g_2\stackrel{x^{-1}}{\rightarrow} g_1$, denoted by $e^{-1}$.
\end{itemize}
All $X$-digraphs in this paper are connected.
A {\em morphism} of two rooted $X$-digraphs is a graph morphism which maps
the root to the root and preserves labels. For more information on $X$-digraphs we refer to
\cite{Stallings:1983,Kapovich_Miasnikov:2002}.

\begin{example}
The {\em Cayley graph} of the group $F/N$, denoted by $\Cay(F/N)$,
is an $X$-digraph $(V,E)$, where $V = F/N$ and
$$
  E=\{ g\stackrel{x}{\rightarrow} gx \mid g\in F/N,\ x\in X^{\pm}\}.
$$
It is an inverse folded complete graph.
We always assume that the trivial element is the root of $\Cay(F/N)$.

Another important example of an $X$-digraph is the {\em Schreier graph}
$\Sch_G(H)$ of a subgroup $H$ of a group $G=F(X)/N$ defined as $(V,E)$:
$$
V = \{Hg \mid g\in G\}
\mbox{ and }
E = \{Hg \stackrel{x}{\rightarrow} Hgx \mid g\in G,\ x\in X^\pm \}.
$$
The coset $H$ is the root of $\Sch_G(H)$.
\qed
\end{example}

Let $\Gamma$ be an inverse $X$-digraph.
Clearly, $(e^{-1})^{-1} = e$. Hence, the set of all edges can be split
into a disjoint union $E = E^+ \sqcup E^-$ satisfying $(E^+)^{-1} = E^-$
and $(E^-)^{-1} = E^+$. The set $E^+$ is called a set of {\em positive edges}
and the set $E^-$ is called a set of {\em negative edges}.

The {\em rank} $r(\Gamma)$ of an inverse $X$-digraph $\Gamma$ is defined as
$|E^+| - |T|,$ where $T$ is any spanning subtree of $\Gamma$.
The fundamental group $\pi_1(\Gamma)$ is the group of labels
of all cycles at the root; it is naturally a subgroup of
$F(X)$ of the rank $r(\Gamma)$ (see \cite{Kapovich_Miasnikov:2002}).

\subsection{Flows on $X$-digraphs}

Let $\Gamma$ be a deterministic inverse $X$-digraph with the root $v$.
A flow on $\Gamma$ is a function $f:E^+(\Gamma) \rightarrow \MZ$
satisfying the following equality $\sigma(v) = 0$, where $\sigma$ is:
$$
\sigma(v)= \sum_{\alpha(e)=v} f(e) - \sum_{\omega(e)=v} f(e)
$$
for all vertices $v\in V(\Gamma)$ except maybe two vertices $s$ and $t$ for which:
$$
\sigma(s) = -\sigma(t) = 1.
$$
The vertex $s$ is called the {\bf source} and the vertex $t$ is called the {\bf sink} of the flow $f$.
If $s$ and $t$ are not defined, then $f$ is called a {\bf circulation}.
In this paper the source is always the root $v$ of $\Gamma$ and, hence, if $\sigma(s)=0$ then the sink
is $v$ as well.

Flows on deterministic connected inverse rooted $X$-digraphs can be defined
by words in $F(X)$ and only by them as follows.
For every word $w\in F(X)$ there exists at most one path $p_w$ in $\Gamma$
with the origin $v$ labeled with $w$, called the {\em trace} of $w$ in $\Gamma$.
If $p_w$ exists, then we can define the flow $\pi_w$ of $w$ on $\Gamma$ which
for every $e\in E^+$ counts the number of times the edge $e$ is traversed minus
the number of times the edge $e^{-1}$ is traversed by $p_w$.
It is also true that for every flow $f$ on $\Gamma$ there exists $w\in F(X)$
such that $f\equiv \pi_w$, see \cite[Lemma 2.5]{Miasnikov_Romankov_Ushakov_Vershik:2010}.

\subsection{Free solvable groups: tools and techniques}
\label{se:free_solvable_gps}

For a free group $F = F^{(0)} = F(X)$ of rank $r$ denote by
$F^{(1)} = [F^{(0)},F^{(0)}]$ the {\em derived} subgroup  of $F$, and by
$F^{(d)} = [F^{(d-1)}, F^{(d-1)}]$ --  the {\em $d$-th derived subgroup}
of $F$, $d\geq 2$. A free solvable group of rank $r$ and class $d$ is defined
as follows:
\begin{itemize}
    \item
$S_{r,0} = F / F^{(0)}$ is a {\em trivial group} of rank $r$,
    \item
$S_{r,1} = F / F^{(1)}$ is a {\em free abelian group} of rank $r$,
    \item
$S_{r,2} = F / F^{(2)}$ is a {\em free metabelian group} of rank $r$, and
    \item
in general, $S_{r,d} = F / F^{(d)}$ is a {\em free solvable group} of rank $r$ and class $d$.
\end{itemize}
In the sequel we usually identify the set $X$ with its canonical images in $S_{r,d}$.
Note that any two consecutive groups in the list above are related to each other:
$S_{r,i} = F/N$ and $S_{r,i+1} = F/[N,N]$, where $N=F^{(i)}$. Hence, naturally,
every general technique for free
solvable groups studies relations between groups of the type $F/[N,N]$ and $F/N$
establishing an inductive step.

One of the most powerful approaches  to study free solvable groups is via the Magnus embedding.
Let $\mathbb{Z} F/N$ be the group ring of $F/N$ with integer coefficients. By
$\gamma:F \rightarrow F/N$ we denote the canonical factorization epimorphism,
as well its linear extension to $\gamma: \mathbb{Z}F \rightarrow \mathbb{Z}F/N$.
Let $T$ be a free (left) $\mathbb{Z}F/N$-module of rank $r$ with a basis $\{t_1, \ldots, t_r\}$.
Then the set of matrices:
$$M(F/N) =
\left \{ \left(
 \begin{array}{ll}
  g & t\\
  0 & 1
  \end{array}
  \right) \bigg{|} g \in F/N, t \in T \right \}
  $$
forms a group with respect to the matrix multiplication. It is easy to see that the group $M(F/N)$
is isomorphic to the restricted wreath product $\MZ^r \rwr F/N$.

\begin{theorem*}[Magnus embedding, \cite{Magnus:1939}]
The homomorphism $\phi:F \rightarrow M(F/N)$ defined by
$$ x_i \stackrel{\phi}{\mapsto} \left(
 \begin{array}{ll}
  x_i^\gamma & t_i\\
  0 & 1
  \end{array}
  \right),
  \ \ \ i = 1, \ldots, r,
  $$
satisfies $\ker \phi = N^\prime$. Therefore, $\phi$ induces a monomorphism
$\phi: F/[N,N] \hookrightarrow M(F/N).$
\end{theorem*}

The Magnus embedding gives a solution to the word problem for free solvable groups.
Using induction on the solvability class $d$ gives a polynomial estimate $O(r^{d-1}|w|^{2d-1})$
on the complexity of the word problem, see \cite[Section 2.2]{Miasnikov_Romankov_Ushakov_Vershik:2010}

Another important technique for studying free solvable groups was introduced
and studied by R. Fox in a sequence of papers \cite{Fox_calc1,Fox_calc2,Fox_calc3,Fox_calc4}
who invented {\em free differential calculus}.
Recall that a {\em free partial derivative} $\tfrac{\partial w}{\partial x_i}$
of the element $w$ of the group $F/[N,N]$ by $x_i$ is an element
of the group ring $\MZ F/N$ given by the formula:
\begin{equation} \label{eq:derivative}
\frac{\partial w}{\partial x_i} = \sum_{1\le j\le n, ~i_j = i,~ \varepsilon_j=1} x_{i_1}^{\varepsilon_1} \ldots x_{i_{j-1}}^{\varepsilon_{j-1}} - \sum_{1\le j\le n, ~ i_j = i,~ \varepsilon_j=-1} x_{i_1}^{\varepsilon_1} \ldots x_{i_{j}}^{\varepsilon_{j}}.
\end{equation}
The following result is one of the principle technical tools in this area, it follows easily from the Magnus embedding theorem, but in the current form it is due to Fox.

\begin{theorem*}[Fox]
Let $N$ be a normal subgroup of $F$ and $\gamma:\MZ F \to \MZ F/N$ the canonical epimorphism.
Then for every $u \in F$ the following equivalence holds:
$$\forall i  \ \left (\partial u/\partial x_i \right )^{\gamma} = 0 \ \ \Leftrightarrow \ \ u \in [N,N].$$
\end{theorem*}

Another approach to study elements of groups $F/[N,N]$ comes from geometric flows on $\Cay(F/N)$.
As we discussed in Section \ref{se:flows} a word  $w \in F(X)$
determines a unique path $p_w$ in $\Cay(F/N)$ labeled by $w$
which starts at the root (the vertex corresponding to the identity of $G$). The path $p_w$
further defines  a geometric flow $\pi_w$ on $\Cay(F/N)$.
Figure \ref{fi:diagram1} gives an example of a flow
for a particular word $w$ in $\Cay(F_2/[F_2,F_2])$.
Nonzero values of $\pi_w$ are shown on the edges and zero values are omitted.
\begin{figure}[htbp]
\centerline{ \includegraphics[scale=1]{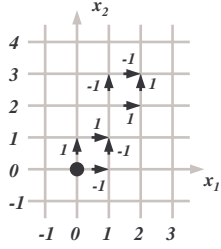} }
\caption{\label{fi:diagram1} The values of $\pi_w$ for $w = x_2
x_1 x_2 x_1 x_2 x_1^{-1} x_2^{-3} x_1^{-1}$ on $(x_1,x_2)$-grid.
In this case  $\partial w/ \partial x_1 = -1 + x_2 -
x_1 x_2^3 + x_1 x_2^2$ and $\partial w/ \partial x_2 = 1 - x_1 +
x_2^2 x_2^2 - x_1 x_2^2$.}
\end{figure}

\begin{theorem*}[\cite{DLS,Vershik_Dobrynin:2004,Miasnikov_Romankov_Ushakov_Vershik:2010}]\label{th:pi}
Let $N$ be a normal subgroup of $F$ and $u,v\in F$. Then $u=v$ in $F/[N,N]$ if and only if
$\pi_u=\pi_v$ in $\Cay(F/N)$.
\qed
\end{theorem*}

\subsection{Computational model and data representation}
\label{se:Computational_Model}

All computations are assumed to be performed on
a random access machine.
(Quasi-)Linear time is very sensitive to the way
one represents the data, so here we describe precisely
how the inputs are given to us.
We use base $2$ positional number system in which presentations of integers
are converted into integers via the rule:
    $$(a_{k-1}\ldots a_3a_2a_1a_0)_2 = a_{k-1}2^{k-1}+\ldots +a_22^2+a_12+a_0,$$
where we assume that $a_{k-1}=1$. The number $k$ is called the {\em bit-length} of the presentation.
\begin{itemize}
    \item
Adding two numbers of bit-length at most $n$ has $O(n)$ time complexity.
The result is a number of bit-length at most $n+1$.
    \item
Computational complexity of multiplying two $n$-bit numbers is $O(n \log n \log\log n)$
(see \cite{Schonhage-Strassen:1971}). The result is a $2n$-bit number.
\end{itemize}

Let $G$ be a group generated by a finite set $X_r = \{x_1,\ldots,x_r\}$.
We formally encode the word problem for $G$ as a subset of $\{0,1\}^\ast$
as follows. We first encode elements of the set $X_r^\pm = \{x_1^\pm,\ldots,x_r^\pm\}$
by unique bit-strings of length $\lceil \log_2 r\rceil+1$.
The code for a word $w=w(X_r^\pm)$ is a concatenation of codes for letters and, formally:
$$
\WP(S_{r,d}) = \{\code(w) \mid w\in F(X_r),\ w=1 \mbox{ in } S_{r,d}\}.
$$
Thus, the bit-length of the representation for a word $w \in F(X_r)$ is:
    $$|\code(w)| = |w| (\lceil \log_2 r\rceil+1).$$
We encode the power and conjugacy problems in a similar fashion.
For both of these problems instances are pairs of words and the encoding
can be done by introducing a new letter ``,'' into the alphabet $X_r^{\pm}$.

Note that any permutation of $X_r$ induces an automorphism
of a free solvable group and taking an automorphic image of a word
preserves the property of being trivial.
Furthermore, for any word $w$ we can find in linear time in $|\code(w)|$
an appropriate automorphic image satisfying $r\le |w|$.
Therefore, we always assume that $r\le |w|$.

\subsection{Quasi-linear time complexity}

An algorithm is said to run in {\em quasi-linear time} if its time complexity function
is $O(n \log^k n)$ for some constant $k\in \MN$.
We use notation $\tO(n)$ to denote quasi-linear time complexity.
Quasi-linear time algorithms are also $o(n^{1+\varepsilon})$
for every $\varepsilon > 0$, and thus run faster than any polynomial
in $n$ with exponent strictly greater than $1$.
See \cite{Naik-Regan-Sivakumar:1995} for more on quasi-linear time complexity theory.
Similarly, one can define quasi-quadratic $\tO(n^2)$, quasi-cubic $\tO(n^3)$ time complexity
as $O(n^2\log^k n)$, $O(n^3\log^k n)$, etc.

\section{The word problem: deterministic solution}
\label{se:WP}

In this section we present a fast deterministic solution
for the word problem in free solvable groups.

\subsection{Support graphs}

Let $\Gamma$ be a rooted folded inverse $X$-digraph and $m$ the length
of a shortest cycle in $\Gamma$. Suppose that a reduced nontrivial
word $w$ can be traced in $\Gamma$. The set of edges traversed by $w$ in
$\Gamma$ forms a connected $X$-digraph called the {\bf support graph}
of $w$ in $\Gamma$.

\begin{lemma}\label{le:trivial_flow_length}
Let $\Gamma$ be a rooted folded inverse $X$-digraph and $m$ the length
of a shortest cycle in $\Gamma$. Suppose that a reduced nontrivial
word $w$ can be traced in $\Gamma$ and $\pi_w = 0$. Then $|w| \ge 3m$.
\end{lemma}

\begin{proof}
It follows from our assumption $\pi_w = 0$ that the path $p_w$ is a cycle in $\Gamma$.
Let $\Delta$ be the support graph of $w$ in $\Gamma$.
The rank of $\Delta$ can not be $0$ ($w$ is not reduced in this case)
and can not be $1$ (either $w$ is not reduced or $\pi_w\ne 0$).
Therefore, the rank of $\Delta$ is at least $2$.
Each edge of $\Delta$ is traversed by $w$ at least twice.
Hence, it is sufficient to prove that $2|E(\Delta)| \ge 3m$.
Let $\Delta'$ be a minimal subgraph of $\Delta$ of rank exactly $2$.
There are exactly two distinct configurations possible for $\Delta'$,
shown in Figure \ref{fi:support_min}.

\begin{figure}[htbp]
\centerline{ \includegraphics[scale=1]{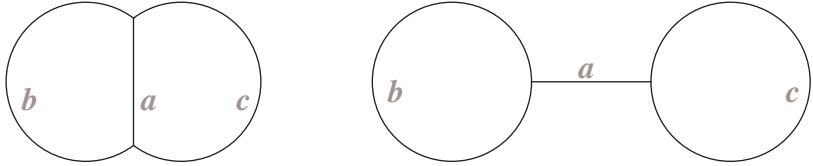} }
\caption{\label{fi:support_min} Two configurations for support graphs in Lemma \ref{le:trivial_flow_length}.}
\end{figure}

Let $a,b,c$ be the lengths of arcs as shown in the figure.
Since, the length of a shortest cycle in $\Gamma$ is $m$, we get the following bounds for our cases:
$$
\begin{array}{llll}
\left\{
\begin{array}{lll}
a+b \ge m, \\
a+c \ge m, \\
b+c \ge m, \\
\end{array}
\right.
&&&
\left\{
\begin{array}{lll}
b \ge m, \\
c \ge m. \\
\end{array}
\right.
\end{array}
$$
In both cases we have $2(a+b+c) \ge 3m$ which proves that $2|E(\Delta)| \ge 3m$.
Thus, $|w| \ge 3m$.
\end{proof}

\begin{proposition}\label{pr:shortest identity}
Let $w \in F(X_r) \setminus \{\varepsilon\}$.
If $w=1$ in $S_{r,d}$, then $|w|\ge 3^d$.
\end{proposition}

\begin{proof}
Induction on $d$. The length of a shortest nontrivial relator in $S_{r,0}$ is $1$.
Assume that the statement holds for all values of solvability class less than $d$.
Hence, the length of a shortest cycle in $\Cay(S_{r,d-1})$ is not smaller than $3^{d-1}$.
Choose a shortest nontrivial relator $w$ in $S_{r,d}$. By Theorem \ref{th:pi}
$w$ defines the trivial flow in $\Cay(S_{r,d-1})$. By Lemma \ref{le:trivial_flow_length}
$|w|\ge 3^d$.
\end{proof}

\subsection{Distinguishers}

First, we fix some notation. For $a,b\in\MZ$ ($a\le b$) define a set:
    $$[a,b] = \{a,a+1\ldots,b-1,b\}.$$
Fix a reduced word $w = x_{i_1}^{\varepsilon_1}\ldots x_{i_k}^{\varepsilon_k} \in F(X_r)$.
For $j\in[0,k]$ by $w_j$ we denote the initial segment of $w$ of length $j$.
By $e_j$ we denote the edge $w_{j-1} \stackrel{x_{i_j}^{\varepsilon_j}}{\rightarrow} w_j$
traversed by $w$ in $\Cay(S_{r,d})$.
By $\Gamma_{d,w}$ we denote the support graph for $w$ in $\Cay(S_{r,d})$.

A word $w$ is trivial in $S_{r,d}$ if and only if
it defines the trivial flow $\pi_w$ on $\Cay(S_{r,d-1})$.
Obviously, the function $\pi_w$ is trivial outside of the support graph $\Gamma_{d-1,w}$ and, hence:
$$
w=1 \mbox{ in } S_{r,d} \ \ \Leftrightarrow \ \ \pi_w\equiv 0 \mbox{ on }\Gamma_{d-1,w}.
$$
Furthermore, for any $j\in[0,k]$ the graph $\Gamma_{d-1,w_j}$ is a subgraph of $\Gamma_{d-1,w}$ and,
hence, we can view $\pi_{w_j}$ as a flow on $\Gamma_{d-1,w}$.
In particular, $w_i=w_j$ in $S_{r,d}$ if and only if
$w_i$ and $w_j$ define the same flows on $\Gamma_{d-1,w}$.
The algorithm described in this section efficiently constructs
graphs $\Gamma_{d,w}$ and flows $\pi_w$ by induction on $d$.
That goal is achieved via the concept of a distinguisher.

\begin{definition}
Let $w=w = x_{i_1}^{\varepsilon_1}\ldots x_{i_k}^{\varepsilon_k} \in F(X_r)$.
We say that a function $\nu_d:[0,k] \rightarrow [0,k]$ is a {\bf distinguisher}
for $w$ in $S_{r,d}$ if it satisfies the following property:
    $$\nu_d(i)=\nu_d(j) \ \ \Leftrightarrow \ \ w_i=w_j \mbox{ in } S_{r,d}.$$
A function $\epsilon_d:[1,k] \rightarrow [-k,k]\setminus\{0\}$
is called an {\bf edge numbering function} for $w$ in $S_{r,d}$ if:
\begin{itemize}
    \item
$\epsilon_d(i)=\epsilon_d(j)$ if and only if $e_i=e_j$;
    \item
$\epsilon_d(i)=-\epsilon_d(j)$ if and only if $e_i=e_j^{-1}$.
\qed
\end{itemize}
\end{definition}

For any function $\nu:[0,k] \rightarrow [0,k]$ one can construct a rooted $X$-digraph
$\Gamma_\nu = (V,E)$, with:
\begin{equation}\label{eq:nu_support}
V = \{\nu(0),\ldots,\nu(k)\}\ \mbox{ and } \ E = \{\nu(j-1) \stackrel{x_j}{\rightarrow} \nu(j) \mid j=1,\ldots,k\}
\end{equation}
with the root at $\nu(0)$.
It is easy to see that
if $\nu_d$ is a distinguisher for $w$ in $S_{r,d}$, then
$\Gamma_{\nu_d}$ is isomorphic to the support graph $\Gamma_{d,w}$
and does not depend on a choice of a distinguisher $\nu_d$.

\begin{lemma}\label{le:numbering_edges}
Given a distinguisher $\nu_d$ for $w$ it requires quasi-linear time
$\tO(|w|)$ to compute an edge-numbering function $\epsilon$ for $w$.
\end{lemma}

\begin{proof}
Each edge is uniquely defined by a triple $(\nu_d(j),\nu_d(j+1),x_{i_j}^{\varepsilon_j})$.
As we explained in Section \ref{se:Computational_Model}, we may assume that $r\le |w|$.
Hence, such triples can be encoded by bit-strings of length $O(\log_2|w|)$.
Organizing a tree of such bit-strings we can sort them and number lexicographically.
Also, it is easy to check if two edges are inverses of each other.
\end{proof}

Our next goal is to construct a sequence of distinguishers $\nu_0,\ldots,\nu_d$
for a given $w$.
Clearly, we can put $\nu_0 \equiv 0$ because $S_{r,0}$ is the trivial group.
Assume that $\nu_{d-1}$ is constructed.
Below we describe a procedure constructing
a distinguisher $\nu_{d}$.

\begin{proposition}\label{pr:computing_distinguisher}
There exists a deterministic quasi-quadratic algorithm which for
a word $w$ and a distinguisher $\nu_{d-1}$ for $w$
produces a distinguisher $\nu_d$.
\end{proposition}

\begin{proof}
Using Lemma \ref{le:numbering_edges} we number the edges of $\Gamma_{d-1,w}$
traversed by $w$ in quasi-linear time.
To construct $\nu_d$ we process $w$ letter by letter constructing flows $\pi_{w_0},\ldots,\pi_{w_k}$.
The function $\pi_{w_i}:E(\Gamma_{d-1,w}) \rightarrow \MZ$ counts the algebraic number of
times each edge is traversed by $w_i$. Since our edges are numbered
we can encode $\pi_{w_i}$'s as tuples $A_i$ of length $|E(\Gamma_{d-1,w})|$.
The function $\pi_{w_0}$ is encoded as the tuple of zeros.
Clearly, $A_{i+1}$ can be obtained from $A_i$ by adding $\pm 1$ to a single component.
Each tuple $A_i$ has length $|E(\Gamma_{d-1,w})| \le |w|$ with absolute values of entries bounded by $|w|$.
Hence, it takes $\tO(|w|)$ time to produce $A_{i+1}$ from $A_{i}$.
Thus, our procedure produces $|w|+1$ bit-strings $A_0,\ldots,A_k$ of length $O(|w|\log_2|w|)$
that uniquely represent the initial segments of $w$ as elements of $S_{r,d}$.
We organize these strings into a tree
and number them according to the lexicographic order.
The obtained numbering gives a required distinguisher $\nu_{d}$.

It is straightforward to construct the tree described above.
The size of the tree is $O(|w|^2 \log_2 |w|)$. Hence, the procedure
has quasi-quadratic time complexity.
\end{proof}

\begin{theorem}\label{th:deterministic_WP}
The word problem in $S_{r,d}$ can be solved by a deterministic procedure
in quasi-quadratic time in $|w|$.
\end{theorem}

\begin{proof}
Using the procedure described in the proof of Proposition \ref{pr:computing_distinguisher}
we compute distinguishers $\nu_0,\ldots,\nu_d$ for $w$. Computation of $\nu_{i+1}$ from $\nu_i$
requires quasi-quadratic time in $|w|$. It follows from Proposition \ref{pr:shortest identity}
that if $d>\log_3|w|$, then $w\ne 1$ in $S_{r,d}$. Hence, we only need to check
the values of $d\le \log_3|w|$. Thus, we only need to compute up to $\log_3|w|$ distinguishers.
This implies that the procedure is quasi-quadratic in $|w|$.
\end{proof}

This gives the first improvement to the algorithm
described in \cite{Miasnikov_Romankov_Ushakov_Vershik:2010}.

\section{Abstract support graphs}
\label{se:graph_support}

In Section \ref{se:WP} we used support graphs to solve the word
problem in free solvable groups in quasi-quadratic time.
In Section \ref{se:randomized_WP} we design a randomized quasi-linear algorithm
for the same problem. To better understand its behavior (to prove that it is false-biased)
we need a notion of an abstract support graph.
The basic idea is to forget that $w$ is traced in $\Cay(S_{r,d})$ and consider any graph ``covered'' by $w$.

For a word $w = x_{i_1}^{\varepsilon_1}\ldots x_{i_k}^{\varepsilon_k} \in F(X)$
define a rooted inverse $X$-digraph
$\Gamma(w) = (V,E)$:
$$
V = \{w_j \mid j\in [0,k]\}
 \ \mbox{ and }\
E = \{w_{j-1} \stackrel{x_{i_j}^{\varepsilon_j}}{\rightarrow} w_j \mid j\in [1,k]\},
$$
where $w_j = x_{i_1}^{\varepsilon_1}\ldots x_{i_j}^{\varepsilon_j}$, see Figure \ref{fi:Gamma_w}.
\begin{figure}[t]
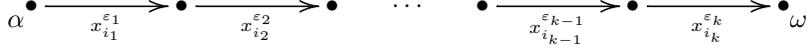

\centerline{
\xygraph{
!{<0cm,0cm>;<2cm,0cm>:<0cm,2cm>::}
!{(0,0)}*+{\bullet}="0"
!{(1,0)}*+{\bullet}="1"
!{(2,0)}*+{\bullet}="2"
!{(3,0)}*+{\bullet}="3"
!{(4,0)}*+{\bullet}="4"
!{(5,0)}*+{\bullet}="5"
"0":@[|(1.5)]"1"_{x_{i_1}^{\varepsilon_1}}
"1":@[|(1.5)]"2"_{x_{i_2}^{\varepsilon_2}}
"3":@[|(1.5)]"4"_{x_{i_{k-1}}^{\varepsilon_{k-1}}}
"4":@[|(1.5)]"5"_{x_{i_k}^{\varepsilon_k}}
!{(2.5,0)}*+{\ldots}="6"
!{(-.1,-.1)}*+{\alpha}
!{(5.1,-.1)}*+{\omega}
}}
\caption{\label{fi:Gamma_w}The graph $\Gamma(w)$.}
\end{figure}
We say that a folded rooted $X$-digraph $\Gamma$ is
a support graph for $w$ if there exists an $X$-digraph epimorphism
$\pi:\Gamma(w)\rightarrow\Gamma$. Note that a morphism $\pi$ is unique for $\Gamma$,
because $\Gamma$ is rooted and folded.
Denote the set of all support graphs for $w$ by $\Omega_w$.
$\Omega_w$ is the set all folded homomorphic images of $\Gamma(w)$.
Hence, it is finite.

Let $\Gamma\in\Omega_w$.
Since every initial segment of $w$ defines a flow on $\Gamma$,
we can define a graph $\iota(\Gamma)$:
$$
V(\iota(\Gamma)) = \{\pi_{w_i} \mid i\in [0,k]\}
 \ \mbox{ and }\
E(\iota(\Gamma)) = \{\pi_{w_{j-1}} \stackrel{x_{i_j}^{\varepsilon_j}}{\rightarrow} \pi_{w_j} \mid j\in [1,k]\}.
$$
It is easy to see that the map $w_i \stackrel{\pi}{\rightarrow} \pi_{w_i}$
defines an epimorphism $\pi:\Gamma(w)\rightarrow \iota(\Gamma)$, i.e., $\iota(\Gamma)\in\Omega_w$.
Hence, the map $\Gamma \mapsto \iota(\Gamma)$ defines a function
$\iota:\Omega_w\rightarrow\Omega_w$.

\begin{proposition}
For any $\Gamma\in\Omega_w$ there exists a (unique) $X$-digraph epimorphism
$\varphi:\iota(\Gamma)\rightarrow \Gamma$. Furthermore the following diagram commutes:
\[
\begindc{0}[10]
\obj(0,5)[A]{$\Gamma(w)$}
\obj(8,5)[B]{$\iota(\Gamma)$}
\obj(8,0)[D]{$\Gamma$}
\mor{A}{B}{}
\mor{A}{D}{}
\mor{B}{D}{$\varphi$}
\enddc
\]
\end{proposition}

\begin{proof}
Vertices of $\iota(\Gamma)$ are flows $\pi_{w_i}$ on $\Gamma$.
Each flow $\pi_{w_i}$ has the sink,
which is the endpoint $\omega(p_{w_i})$ of the path $p_{w_i}$ in $\Gamma$.
Hence, we can define a map $\varphi:V(\iota(\Gamma)) \rightarrow V(\Gamma)$ by
$\pi_{w_i} \stackrel{\varphi}{\mapsto} \omega(p_{w_i}).$
It is easy to check that $\varphi$ is an $X$-digraph morphism satisfying
    $$\pi_{w_i}\equiv\pi_{w_j} \ \Rightarrow\  \omega(p_{w_i})=\omega(p_{w_j})$$
Therefore, the diagram indeed commutes.
\end{proof}

\begin{remark}
Let $H = \pi_1(\Gamma)$. The reader can recognize $\iota(\Gamma)$
as the image of $\Gamma(w)$ in the Schreier graph of the subgroup $[H,H]\le F(X)$.
\qed
\end{remark}

The next proposition shows that a sequence of applications of
$\iota$ always ends up with $\Gamma(w)$.

\begin{lemma}\label{le:iota_power}
For any $\Gamma\in \Omega_w$ we have $\iota^{\log_3|w|} (\Gamma) = \Gamma(w)$.
\end{lemma}

\begin{proof}
If $\Gamma = \Gamma(w)$, then there is nothing to prove. Let $m$ be the length of a shortest cycle in $\Gamma$.
By Lemma \ref{le:trivial_flow_length}, the length of a shortest cycle in $\iota(\Gamma)$ is not smaller than $3m$.
Therefore, $\iota^{\log_3|w|} (\Gamma)$ has no cycles, i.e., $\iota^{\log_3|w|} (\Gamma) = \Gamma(w)$.
\end{proof}

\begin{lemma}\label{le:same_flows}
Let $\Gamma,\Delta \in\Omega_w$ and $\varphi:\Gamma\rightarrow\Delta$ be a rooted $X$-digraph morphism.
Assume that $u$ and $v$ can be traced in $\Gamma$ and define equal flows on $\Gamma$.
Then they define equal flows on $\Delta$.
Therefore, if $w$ defines the trivial flow on
$\Gamma$, then it defines the trivial flow on $\Delta$.
\end{lemma}

\begin{proof}
Let $\pi_u^{(\Gamma)},\pi_v^{(\Gamma)}:E(\Gamma)\rightarrow\MZ$ and
$\pi_u^{(\Delta)},\pi_v^{(\Delta)}:E(\Delta)\rightarrow\MZ$
be flows defined by $u$ and $v$ in $\Gamma$ and in $\Delta$ respectively.
Then for an arbitrary $e\in E(\Delta)$:
    $$\pi_u^{(\Delta)}(e) = \sum_{f\in E(\Gamma),\ \varphi(f)=e} \pi_u^{(\Gamma)}(f) = \sum_{f\in E(\Gamma),\ \varphi(f)=e} \pi_v^{(\Gamma)}(f) = \pi_v^{(\Delta)}(e).$$
Hence, $\pi_u^{(\Delta)} \equiv \pi_v^{(\Delta)}$.
By the same formula $\pi_u^{(\Gamma)} \equiv 0$ implies $\pi_u^{(\Delta)} \equiv 0$.
\end{proof}

\begin{proposition}\label{pr:main_diagram}
Let $\Gamma,\Delta \in\Omega_w$ and $\varphi:\Gamma\rightarrow\Delta$
be an epimorphism. Then there exists an epimorphism $\psi:\iota(\Gamma)\rightarrow\iota(\Delta)$
such that the diagram below commutes:
\[
\begindc{0}[10]
\obj(0,5)[A]{$\iota(\Gamma)$}
\obj(8,5)[B]{$\Gamma$}
\obj(0,0)[C]{$\iota(\Delta)$}
\obj(8,0)[D]{$\Delta$}
\mor{A}{B}{$\tau_1$}
\mor{A}{C}{$\psi$}
\mor{B}{D}{$\varphi$}
\mor{C}{D}{$\tau_2$}
\enddc
\]
\end{proposition}

\begin{proof}
By definition $V(\iota(\Gamma)) = \{\pi_{w_i}^{(\Gamma)} \mid i\in[0,k]\}$
and $V(\iota(\Delta)) = \{\pi_{w_i}^{(\Delta)} \mid i\in[0,k]\}$. The map
$\psi:V(\iota(\Gamma)) \rightarrow V(\iota(\Delta))$ given by:
    $$\pi_{w_i}^{(\Gamma)} \ \stackrel{\psi}{\mapsto} \ \pi_{w_i}^{(\Delta)},$$
is well defined by Lemma \ref{le:same_flows}.
The map $\psi$ takes an edge
$\pi_{w_{i-1}}^{(\Gamma)} \stackrel{x_i}{\rightarrow} \pi_{w_i}^{(\Gamma)}$
in $\Gamma$ to the edge
$\pi_{w_{i-1}}^{(\Delta)} \stackrel{x_i}{\rightarrow} \pi_{w_i}^{(\Delta)}$
in $\Delta$. Therefore, $\psi$
preserves connectedness and labels and is indeed
an $X$-digraph epimorphism.

Finally we note that for any $\pi_{w_{i-1}}^{(\Gamma)} \in V(\iota(\Gamma))$
we have
$\tau_1(\pi_{w_{i-1}}^{(\Gamma)})$ is the endpoint of $w_i$ traced in $\Gamma$.
Similarly, $\tau_2(\psi(\tau_1(\pi_{w_{i-1}}^{(\Gamma)})))$ is the endpoint
of $w_i$ traced in $\Delta$. Since, $\varphi$ is an $X$-digraph morphism
taking the root to the root, we have a commuting diagram.
\end{proof}

\subsection{Language support graphs}
\label{se:language_support}

Definition of a word support graph can be generalized to any set of words
$S\subseteq F(X)$ as follows. Define a {\em prefix tree} $T=T(S)$:
$$
V(T) = \{u \in F(X) \mid u\circ v \in S\} \ \mbox{ and }\  E(T) = \{w \stackrel{x}{\rightarrow} wx \mid w,wx\in S,\ x\in X^\pm\}.
$$
We say that a rooted inverse $X$-digraph $\Gamma$ is a {\em support graph}
for $S$ if there exists an $X$-digraph epimorphism $T\rightarrow \Gamma$.
Assume that $S$ is finite.
The (finite) set of all support graphs for $S$ is denoted by $\Omega_S$.
For any $\Gamma\in\Omega_S$ we can define the graph $\iota(\Gamma) =  (V,E)$:
$$
V(\iota(\Gamma)) = \{\pi_{w} \mid w\in V(T(S))\}
 \ \mbox{ and }\
E(\iota(\Gamma)) = \{\pi_{w} \stackrel{x}{\rightarrow} \pi_{wx} \mid w,wx \in V(T(S))\}.
$$
It easy to check that all results in this section for word support graphs hold for
language support graph as well. Lemma \ref{le:iota_power} can be reformulated as follows:

\begin{lemma}\label{le:iota_power_T}
Let $S$ be a finite subset of $F(X_r)$ and $d$ the diameter of the tree $T(S)$.
Then for any $\Gamma\in \Omega_S$ we have $\iota^{\log_3 d} (\Gamma) = T(S)$.
\qed
\end{lemma}

\section{The word problem: randomized solution}
\label{se:randomized_WP}

In this section we improve quasi-quadratic procedure
described in Proposition \ref{pr:computing_distinguisher}, we make it quasi-linear.
Let $w=x_{i_1}^{\varepsilon_1}\ldots x_{i_k}^{\varepsilon_k} \in F(X)$,
$\Gamma_{d-1}$ be the support graph for $w$ in $S_{r,d-1}$, and $m = |E(\Gamma_{d-1})|$.
Recall that the algorithm computes the set of tuples $A_0,\ldots,A_k\in\MZ^m$
that represent the flows $\pi_{w_0},\ldots,\pi_{w_k}$ on $\Gamma_{d-1}$.
Tuples $A_i$ are further encoded as bit-strings of lengths $O(|w| \log_2|w|)$).
Hence, we deal with $|w|$ objects of size $|w|$ which makes quadratic complexity.

To improve complexity we choose a point $A(a_1,\ldots,a_m) \in [0,|w|^3]^m$ uniformly randomly
and replace $A_i$'s with the numbers $d^2(A,A_i)$, where $d$ is the Euclidean distance in $\MZ^m$.
Those numbers have bit-lengths $O(\log_2|w|)$.
The next two lemmas are concerned with complexity of computing the numbers $d^2(A,A_i)$.

\begin{lemma}\label{le:initial_distance}
It requires $\tO(|w|)$ time to compute $d^2(A,A_0)=\sum_{i=1}^m a_i^2$.
\end{lemma}

\begin{proof}
For every $i=1,\ldots,|w|$ the bit-length of $a_i$ is bounded by $3\log |w|$.
Sch\"onhage-Strassen algorithm requires $O(\log |w| \log\log |w| \log\log\log |w|)$
steps to compute each $a_i^2$. The bit-length of $a_i^2$ is bounded by $6\log |w|$.
Finally, it requires $O(|w|\log |w|)$ steps to sum $m$
obtained squares each of length $6\log |w|$.
Thus, the total complexity is $O(|w|\log |w| \log\log |w| \log\log\log |w|)$.
\end{proof}

Let $\overline{AA}_i = (A_{i,1},\ldots,A_{i,m})$. The vectors
$\overline{AA}_i$ and $\overline{AA}_{i+1}$ differ at a single, say $j$th, component
and $|A_{i,j}-A_{i+1,j}|=1$. Therefore,
\begin{equation}\label{eq:distance_Ai}
d^2(A,A_{i+1})-d^2(A,A_i) =
\begin{cases}
2|A_{i,j}|+1 & \mbox{if } |A_{i+1,j}|>|A_{i,j}|,\\
-2|A_{i,j}|+1 & \mbox{if } |A_{i+1,j}|<|A_{i,j}|.\\
\end{cases}
\end{equation}

\begin{lemma}\label{le:distance_bound}
Given $\overline{AA}_i$ and $d^2(A,A_i)$ it takes $O(\log|w|)$
time to compute $\overline{AA}_{i+1}$ and $d^2(A,A_{i+1})$.
\end{lemma}

\begin{proof}
First note that
    $$d^2(A,A_i) \le \sum_{i=1}^m (|w|^3+|w|)^2 \le (|w|^3+|w|)^2|w|.$$
Hence, $d^2(A,A_i)$ has bit-length $O(\log|w|)$.
Similarly, $\pm 2|A_{i,j}|+1$ has bit-length $O(\log|w|)$.
It requires $O(\log|w|)$ to compute $\pm 2|A_{i,j}|+1$.
Finally, it takes the same time to take the sum of $d^2(A,A_i)$ and $\pm 2|A_{i,j}|+1$.
\end{proof}

\begin{proposition}\label{pr:computing_randomized_distinguisher}
There exists a randomized quasi-linear algorithm which given
a word $w = x_{i_1}^{\varepsilon_1}\ldots x_{i_k}^{\varepsilon_k} \in F(X_r)$
and a distinguisher $\nu_{d-1}$ for $w$
produces a function $\nu_d$. The function $\nu_d$ is a distinguisher for $w$
with the probability at least $1-\tfrac{1}{|w|}$, where the probability
is taken over all uniform choices of the point $A$ in $[0,|w|^3]^m$.
\end{proposition}

\begin{proof}
Using Lemmas \ref{le:initial_distance} and \ref{le:distance_bound}
we can compute the array $i\mapsto d^2(A,A_i)$ in $\tO(|w|)$ time.
The numbers $\{d(A,A_i)\}_{i=0}^{|w|}$ are then lexicographically sorted and numbered
$d^2(A,A_i) \mapsto n_i$. The function $\nu_d$ is the composition $i\mapsto n_i$.
Overall, it requires in $\tO(|w|)$ time to compute $\nu_d$.

The described algorithm makes a mistake
when $d^2(A,A_i) = d^2(A,A_j)$ while $A_i\ne A_j$ for some $i,j\in [0,|w|]$.
This happens only when the randomly chosen point $A$ belongs to a hyperplane
with a normal vector $\overline{A_iA_j}$ through $\tfrac{1}{2}(A_i+A_j)$.
The union of hyperplanes for all pairs of points $(A_i,A_j)$ contains at most $\tfrac{1}{|w|}$
part of the hypercube $[0,|w|^3]^m$. Hence, the uniformly chosen $A$ has
not greater than $\tfrac{1}{|w|}$ chance to collapse two distinct points $A_i,A_j$.
\end{proof}

The next proposition states that the support graph of the function produced by our algorithm
is a homomorphic image of the correct support graph $\Gamma_{d,w}$.

\begin{proposition}\label{pr:mistake}
Let $\nu_{d-1},\nu_d$ be distinguishers for $w$,  $\nu_d'$ be a
function produced from $\nu_{d-1}$ by the randomized
procedure described in Proposition \ref{pr:computing_randomized_distinguisher}, and
$\Delta = \Gamma_{\nu_d'}$ (see formulae (\ref{eq:nu_support})).
Then there exists an epimorphism $\varphi:\Gamma_d\rightarrow\Delta$
which is an isomorphism if and only if the algorithm does not make
a mistake.
\end{proposition}

\begin{proof}
For any $i,j\in[0,k]$:
    $$\nu_d(i)=\nu_d(j) \ \ \Leftrightarrow\ \ A_i=A_j\ \ \Rightarrow\ \ \nu_d'(i)=\nu_d'(j).$$
Therefore, there exists a (unique) epimorphism $\varphi:\Gamma_d\rightarrow\Delta$.
Clearly, $\varphi$ is an isomorphism if and only if
$$\nu_d(i)=\nu_d(j) \ \ \Leftrightarrow\ \ \nu_d'(i)=\nu_d'(j),$$
i.e., when the algorithm outputs a correct distinguisher.
\end{proof}

\begin{theorem}\label{th:randomized_WP}
Let $r,d\in\MN$ and $w\in F(X_r)$.
There exists a quasi-linear randomized algorithm deciding if $w=1$ in $S_{r,d}$, or not.
Furthermore,
\begin{itemize}
    \item[(a)]
if $w=1$ in $S_{r,d}$, then the algorithm outputs $Yes$;
    \item[(b)]
if $w\ne 1$ in $S_{r,d}$, then the algorithm outputs $No$
with probability at least $\rb{1-\tfrac{1}{|w|}}^{\log_3 |w|}$.
\end{itemize}
\end{theorem}

\begin{proof}
Starting with $\nu_0\equiv 0$ we compute distinguishers $\nu_0',\ldots,\nu_d'$
using the randomized algorithm described in
the proof of Proposition \ref{pr:computing_randomized_distinguisher}.
Output $Yes$ if $\nu_d'(|w|)=\nu_d'(0)$. Otherwise, output $No$.
Since $d$ can be bounded by $\log_3|w|$ the described algorithm
has quasi-linear complexity in $|w|$.

Assume that $w=1$ in $S_{r,d}$.
Let $\nu_1',\ldots,\nu_d'$ be the sequence of functions inductively
produced by the randomized algorithm. By construction we have $\nu_0'=\nu_0 \equiv 0$.
Denote $\Gamma_{\nu_i'}$ by $\Delta_i$. It follows from
Propositions \ref{pr:mistake} and \ref{pr:main_diagram} that for every $i\in [0,k]$
there exists an epimorphism $\tau_i:\Gamma_{i} \rightarrow \Delta_i$.
Since $w$ is trivial in $S_{r,d}$, then it has the trivial flow in $\Gamma_{d-1}$.
Hence, by Lemma \ref{le:same_flows}, $w$ also has the trivial flow in $\Delta_{d-1}$.
Therefore, the algorithm outputs $Yes$.

We compute up to $\log_3|w|$ distinguishers.
By Proposition \ref{pr:computing_randomized_distinguisher}, the chance to make a mistake at
any stage is not greater than $\tfrac{1}{|w|}$.
Thus, the chance to make no mistakes is not smaller than $\rb{1-\tfrac{1}{|w|}}^{\log_3 |w|}$.
\end{proof}

Finally, we want to make several remarks on the performance of the algorithm.
The bound in Theorem \ref{th:randomized_WP}(b) can be simplified as follows:
$$
P(\mbox{success}) \ge
\rb{1-\tfrac{1}{|w|}}^{\log_3 |w|}
\ge
1-\tfrac{\log_3 |w|}{|w|}.
$$
Thus, the failure rate of the algorithm decreases almost linearly with $|w|$.

The success rate of the algorithm can be improved by sampling the point
$A$ from a bigger hypercube. For instance, taking numbers uniformly from $[0,|w|^4]$
improves the correctness estimate in Proposition \ref{pr:computing_randomized_distinguisher}
to $1-\tfrac{1}{|w|^2}$ and the correctness estimate in Theorem \ref{th:randomized_WP}
to $1-\tfrac{\log_3 |w|}{|w|^2}$. At the same time bit-lengths of $d^2(A,A_i)$ increase only by
a constant factor leaving the quasi-linear complexity bound intact.

The actual correctness probability is probably much better than our estimates.
Making a mistake on some intermediate step does not imply that the algorithm
will output $Yes$ on $w\ne 1$. In fact, it is possible to get correct
distinguisher $\nu_{i+1}$ starting from incorrect $\nu_{i}$.

\section{The power problem}

In this section we describe the algorithm for
solving the power problem in $S_{r,d}$.
The algorithm is based on two observations.
The first observation is:
    $$u=v^k \mbox{ in } S_{r,d} \ \ \Rightarrow \ \ u=v^k \mbox{ in }S_{r,d-1}.$$
The second observation is
the Malcev theorem on centralizers in free solvable groups.

\begin{theorem*}[\cite{Malcev:1960}]
The centralizer of any nontrivial element $u$ of a free solvable group is abelian. Furthermore:
\begin{itemize}
    \item[(a)]
If $u=1$ in $S_{r,d-1}$, then $[u,v]=1$ in $S_{r,d}$  if and only if $v=1$ in $S_{r,d-1}$.
    \item[(b)]
If $u\ne1$ in $S_{r,d-1}$,
then $[u,v]=1$ in $S_{r,d}$  if and only if $u$ and $v$ are powers of the same element in $S_{r,d}$.
\qed
\end{itemize}
\end{theorem*}

The next algorithm solves the power problem in $S_{r,d}$.

\begin{algorithm}{Power problem in $S_{r,d}$}
\label{al:PP}
\begin{algorithmic}[1]
\REQUIRE $r,d\in\MN$ and $u,v\in F(X_r)$.
\ENSURE $k\in\MZ$ such that $u=v^k$ in $S_{r,d}$ and $Fail$ if $k$ does not exist.
\STATE Let $D = 1+\min\{d,\log_3(|u|+|v|)\}$.
\STATE Construct a sequence of support graphs $\{\Gamma_i\}_{i=0}^{D}$ for the set $\{u,v,[u,v]\}$.
\STATE Use $\{\Gamma_i\}_{i=0}^{D}$ to compute the largest $s$ such that $u=1$ in $S_{r,s}$.
\STATE Use $\{\Gamma_i\}_{i=0}^{D}$ to compute the largest $t$ such that $v=1$ in $S_{r,t}$.
\LINEIF{$d\le s,t$}{\algorithmicreturn\ $1$.}
\LINEIF{$s<d\le t$}{\algorithmicreturn\ $Fail$.}
\LINEIF{$t<d\le s$}{\algorithmicreturn\ $0$.}
\LINEIF{$s<t<d$}{\algorithmicreturn\ $Fail$.}
\LINEIF{$t<s<d$}{\algorithmicreturn\ $Fail$.}
\IF{$s=t<d$}
    \STATE Choose any edge $e$ in $\Gamma_s$ with $\pi_v(e) \ne 0$ in $\Gamma_s$ and put $k=\pi_u(e)/\pi_v(e)$.
    \LINEIF{$k\notin\MZ$ or $[u,v]\ne 1$ in $S_{r,d}$}{\algorithmicreturn\ $Fail$.}\label{st:PP_comm}
    \FORALL{$e$ in $\Gamma_s$}\label{st:PP_for1}
        \LINEIF{$\pi_u(e) \ne k \pi_v(e)$}{\algorithmicreturn\ $Fail$.}
    \ENDFOR\label{st:PP_for2}
    \RETURN $k$.
\ENDIF
\end{algorithmic}
\end{algorithm}

A few details are in order. By Lemma \ref{le:iota_power_T} the support graph
for $T = T(\{u,v,[u,v]\})$ in $S_{r,1+\log_3(|u|+|v|)}$ is $T$ itself
because the diameter of the graph $T$ is not greater than $3(|u|+|v|)$.
In particular, $s,t \le 1+\log_3(|u|+|v|)$.
That explains the choice of $D$.

Algorithm \ref{al:CP} can be implemented as a deterministic or a randomized algorithm
depending on how we compute the sequence of graphs $\{\Gamma_i\}_{i=0}^{D}$.
Computing $\{\Gamma_i\}_{i=0}^{D}$ using the deterministic algorithm from
Theorem \ref{th:deterministic_WP} gives the deterministic version of Algorithm \ref{al:PP}.

\begin{theorem}\label{th:deterministic_PP}
The deterministic Algorithm \ref{al:PP} solves the power problem in $S_{r,d}$
in quasi-quadratic time $\tO((|u|+|v|)^2)$.
\end{theorem}

\begin{proof}
All cases considered in the algorithm are trivial except, maybe,
the case when $s=t<d$. In that case $u\ne 1$ and $v\ne 1$ in $S_{r,d}$.
The flows $\pi_u$ and $\pi_v$ are circulations in $\Cay(S_{r,s})$
(the source and the sink are the same) and $v=u^k$ implies that $\pi_v \equiv k\pi_u$.
Hence, $\pi_u(e)/\pi_v(e)$ is the only possible value for $k$ (if $\pi_v(e)\ne 0$).
Now, we have two cases as in the Malcev theorem.
If $s=t=d-1$, then it is sufficient to check if $\pi_v \equiv k\pi_u$
(done in lines \ref{st:PP_for1}--\ref{st:PP_for2}).
Otherwise, it is sufficient to check if $[u,v]=1$ (part (b) of the Malcev theorem)
which is done in line \ref{st:PP_comm}.
By Theorem \ref{th:deterministic_WP} it takes quasi-quadratic time
to construct support graphs for $T(\{u,v,[u,v]\})$ and test the equality $[u,v]=1$.
\end{proof}

Computing the sequence $\{\Gamma_i\}_{i=0}^{D}$
can also be done using the randomized algorithm from Theorem \ref{th:randomized_WP}.
To obtain the desired probability of success we choose the random tuple $A$
with components chosen uniformly from $[0,9(|u|+|v|)^3]$.

\begin{theorem}\label{th:randomized_PP}
The randomized Algorithm \ref{al:PP} solves the power problem
in $S_{r,d}$ in quasi-linear time $\tO(|u|+|v|)$.
The algorithm returns a correct answer with probability at least:
    $$\rb{1-\tfrac{1}{|u|+|v|}}^{1+\log_3 (|u|+|v|)}.$$
\end{theorem}

\begin{proof}
The complexity estimate immediately follows from Theorem \ref{th:randomized_WP}.
We argue as in Proposition \ref{pr:computing_randomized_distinguisher} to get the correctness lower-bound.
The graph $\Gamma(\{u,v,[u,v]\})$ has at most $3(|u|+|v|)$ vertices
which defines at most $9(|u|+|v|)^2$ bad hyperplanes.
The union of those hyperplanes can contain at most $\tfrac{1}{|u|+|v|}$ part of our hypercube.
Hence, our procedure produces the correct support graph $\Gamma_i$ for $\Gamma(\{u,v,[u,v]\})$ in $S_{r,i}$
with probability at least $1-\tfrac{1}{|u|+|v|}$.
We perform up to $1+\log_3(|u|+|v|)$ iterations.
Hence the claimed correctness probability.
\end{proof}

Algorithm \ref{al:PP} is unbiased, i.e., it can make an error on both positive and negative
instances of the problem.

\section{The conjugacy problem}

In this section we revisit algorithmic difficulty of the conjugacy problem in free solvable groups.
In \cite{Matthews:1966} Matthews  proved that the conjugacy problem (CP) is solvable in wreath products
$A \rwr B$ (under some natural assumptions on $A$ and $B$). She used that result to prove that
CP in free metabelian groups is decidable.
Kargapolov and Remeslennikov generalized that result to free solvable groups
in \cite{Kargapolov-Rem:1966}.
A few years later Remeslennikov and Sokolov in \cite{Remeslennikov-Sokolov:1970} described precisely the
image of $F/[N,N]$ under the Magnus embedding and showed that two elements are conjugate in $S_{r,d}$
if and only if their images are conjugate in $M(F/N)$.
Recently Vassileva in \cite{Vassilieva:2011} found a polynomial time $O(rd(|u|+|v|)^8)$
algorithm for the conjugacy problem combining the Matthews and Remeslennikov-Sokolov results.

\subsection{Matthews algorithm for wreath products}

In this section we shortly outline computations in the proof of the Matthews theorem on conjugacy
in wreath products. Note that we use different notation for wreath products
than Matthews and at the end we obtain slightly different formula.

Let $A,B$ be finitely generated groups.
By $A^B$ we denote the set of all functions $f:B\rightarrow A$ with finite support.
For $f\in A^B$ and $\alpha\in B$ define $f^\alpha \in A^B$ as follows:
$$
f^\alpha(x) = f(\alpha^{-1}x).
$$
The {\em restricted wreath product} of $A$ and $B$, denoted by $A \rwr B$, is a set of pairs:
    $$\{(f,\alpha) \mid f\in A^B,\ \alpha\in B\}$$
with multiplication given by:
$$
(f,\alpha)\cdot(g,\beta) = (f g^\alpha,\alpha\beta).
$$
Hence, for $x=(f,\alpha)$, $y=(g,\beta)$, and $z=(h,\gamma)$ in $A \rwr B$  we have:
$$
zx=yz
\ \ \Leftrightarrow\ \
\left\{
\begin{array}{ll}
\gamma \alpha=\beta \gamma & \mbox{in } B,\\
f^\gamma = h^{-1} g h^\beta & \mbox{in } A^B.
\end{array}
\right.
$$
Assuming the equality above we observe that for any $\delta \in B$ and $j\in\MZ$:
\begin{align*}
f(\gamma^{-1} \beta^j \delta) &= f^\gamma(\beta^j \delta) \\
&= h^{-1}(\beta^j \delta) g(\beta^j \delta) h^\beta(\beta^j \delta) \\
&= h^{-1}(\beta^j \delta) g(\beta^j \delta) h(\beta^{j-1} \delta)
\end{align*}
and, hence, the following formula holds for any $m\in\MZ$ and $n\in\MN$:
\begin{equation}\label{eq:long_cong_prod}
  \prod_{j=m}^{m-n} f^\gamma(\beta^j \delta) =
  h^{-1}(\beta^m \delta) \cdot \rb{ \prod_{j=m}^{m-n}  g(\beta^j \delta)} \cdot h(\beta^{m-n-1} \delta).
\end{equation}
Now, for $\beta,\gamma,\delta \in B$ and $f:B \rightarrow A$
define $\rho^{(f,\beta,\gamma)}(\delta) \in A$ as follows:
$$
\rho^{(f,\beta,\gamma)}(\delta) =
\begin{cases}
\prod_{j=k-1}^{0} f(\gamma^{-1} \beta^j \delta) & \mbox{if } \ord(\beta)=k,\\
\prod_{j=\infty}^{-\infty} f(\gamma^{-1} \beta^j \delta) & \mbox{if } \ord(\beta)=\infty.\\
\end{cases}
$$
Hence, assuming that $zx=yx$ and using the defined above notation,
equality (\ref{eq:long_cong_prod}) gives:
\begin{itemize}
    \item
if $\ord(\beta) < \infty$, then $\rho^{(f,\beta,\gamma)}(\delta)$ is conjugate
to $\rho^{(g,\beta,1)}(\delta)$ in $A$ for every $\delta\in B$;
    \item
if $\ord(\beta) = \infty$, then $\rho^{(f,\beta,\gamma)}(\delta) = \rho^{(g,\beta,1)}(\delta)$
in $A$ for every $\delta\in B$.
\end{itemize}
Matthews proved in \cite{Matthews:1966} that the converse also holds.

\begin{theorem}[\cite{Matthews:1966}]\label{th:Matthews}
Let $A$ and $B$ be finitely generated groups.
Elements $x=(f,\alpha)$ and $y=(g,\beta)$ are conjugate in $A\rwr B$ if and only if there exists
$\gamma \in B$ satisfying:
\begin{itemize}
    \item[(a)]
$\gamma \alpha=\beta \gamma$ in $B$;
    \item[(b)]
if $\ord(\beta) < \infty$, then $\rho^{(f,\beta,\gamma)}(\delta)$ is conjugate
to $\rho^{(g,\beta,1)}(\delta)$ in $A$ for every $\delta\in B$;
    \item[(c)]
if $\ord(\beta) = \infty$, then $\rho^{(f,\beta,\gamma)}(\delta) = \rho^{(g,\beta,1)}(\delta)$
in $A$ for every $\delta\in B$.
\qed
\end{itemize}
\end{theorem}

Recall that the Magnus embedding embeds $S_{r,d}$ into $\MZ^r \rwr S_{r,d-1}$.
The Matthews theorem gives a solution to the conjugacy problem in $\MZ^r \rwr S_{r,d-1}$
and Remeslennikov-Sokolov prove that elements are conjugate in $S_{r,d}$
if and only if their images are conjugate in $\MZ^r \rwr S_{r,d-1}$.
This solves the conjugacy in $S_{r,d}$ and concludes the general algorithm description.

\subsection{Geometric approach to conjugacy problem in free solvable groups}

In the case of a free solvable group the Matthews result can be formulated
in a geometric way using flows on Schreier graphs. By
$\Sch_d(y)$ we denote the Schreier graph of $\gp{y}$ in $S_{r,d}$.
The next lemma follows from the definition of $\rho$.

\begin{lemma}\label{le:rho_pi}
Let $\phi: S_{r,d} \rightarrow \MZ^r \rwr S_{r,d-1}$ be the Magnus embedding.
Let $x,y,z \in F(X_r)$ and $\phi(x) = (f,\alpha)$, $\phi(y) = (g,\beta)$, $\phi(z) = (h,\gamma)$.
Then for any $\delta \in S_{r,d-1}$:
\begin{itemize}
    \item[(a)]
$\rho^{(g,\beta,1)}(\delta) \in \MZ^r$ is exactly the value of $\pi_{y}$
in $\Sch_{d-1}(y)$ restricted to the edges $\delta\rightarrow \delta x_i$ for $i=1,\ldots,r$;
    \item[(b)]
$\rho^{(f,\beta,\gamma)}(\delta) \in \MZ^r$ is exactly the value of $\pi_{\gamma x \gamma^{-1}}$
in $\Sch_{d-1}(y)$ restricted to edges $\delta\rightarrow \delta x_i$ for $i=1,\ldots,r$.
\qed
\end{itemize}
\end{lemma}

\begin{lemma}\label{le:trivial_pi_Sch}
For any $y\in F(X_r)$ we have $\pi_y\equiv 0$ in $\Sch_{d-1}(y)$ if and only if $y=1$ in $S_{r,d}$.
\end{lemma}

\begin{proof}
The equality $\pi_y \equiv 0$, by Lemma \ref{le:rho_pi}, implies $\rho \equiv 0$.
Hence, by Theorem \ref{th:Matthews}, $y\sim 1$ in $S_{r,d}$, i.e., $y=1$ in $S_{r,d}$.

Conversely, if $y=1$ in $S_{r,d}$, then $y=1$ in $S_{r,d-1}$. Hence, $\Sch_{d-1}(y) = \Cay(S_{r,d-1})$
and $\pi_y \equiv 0$ in $\Cay(S_{r,d-1})$.
Hence, sufficiency holds.
\end{proof}

The next theorem states that $x\sim y$ in $S_{r,d}$ if and only if
there exists a shift of $x$ in $\Sch_{d-1}(y)$ defining the same flow
as $y$.

\begin{theorem}\label{th:geometric_conjugacy}
Words $x,y \in F(X_r)$ represent conjugate elements in $S_{r,d}$
if and only if there exists $z\in F(X)$
satisfying $\pi_{zxz^{-1}} \equiv \pi_{y}$ in $\Sch_{d-1}(y)$.
The element $z$ can be viewed as an element of $S_{r,d-1}$.
\end{theorem}

\begin{proof}
If $x$ and $y$ represent conjugate elements in $S_{r,d}$, then for some
$z$ we have $zxz^{-1}=y$ in $S_{r,d}$. Hence, $\pi_{zxz^{-1}} \equiv \pi_y$
in $\Cay(S_{r,d-1})$ and in $\Sch_{d-1}(y)$. Thus, necessity holds.

Conversely, assume that there exists $z\in F(X)$
satisfying $\pi_{zxz^{-1}} \equiv \pi_{y}$ in $\Sch_{d-1}(y)$.
If $\pi_y \equiv 0$ in $\Sch_{d-1}(y)$, then, by Lemma \ref{le:trivial_pi_Sch},
$y=1$ in $S_{r,d}$ and $\Sch_{d-1}(y) = \Cay(S_{r,d-1})$. Hence,
$\pi_{zxz^{-1}}\equiv 0$ in $\Cay(S_{r,d-1})$ and $z=1$ in $S_{r,d}$ as well.

Hence, we may assume that $\pi_y\not\equiv 0$ in $\Sch_{d-1}(y)$.
Now the equality $\pi_{zxz^{-1}} \equiv \pi_{y}$ in $\Sch_{d-1}(y)$ implies that
$zxz^{-1}$ is a cycle in $\Sch_{d-1}(y)$, i.e., $zxz^{-1}$ is a power of $y$.
Furthermore, since $\pi_y\not\equiv 0$, we clearly have $zxz^{-1} = y$ in $S_{r,d-1}$.
Thus, item (a) of Theorem \ref{th:Matthews} holds.

Let $\phi(x) = (f,\alpha)$, $\phi(y) = (g,\beta)$, $\phi(z) = (h,\gamma)$.
By Lemma \ref{le:rho_pi} the equality  $\pi_{zxz^{-1}} \equiv \pi_{y}$ implies that
$\rho^{(f,\beta,\gamma)}(\delta) = \rho^{(g,\beta,1)}(\delta)$ for every $\delta\in S_{r,d-1}$.
Hence, both items (b) and (c) of Theorem \ref{th:Matthews} hold.
\end{proof}

Now it is straightforward to solve the conjugacy problem in $S_{r,d}$.

\begin{algorithm}{Conjugacy problem in $S_{r,d}$}
\label{al:CP}
\begin{algorithmic}[1]
\REQUIRE $r,d\in\MN$ and $x,y\in F(X_r)$.
\ENSURE $Yes$ if $u\sim v$, $No$ otherwise.
\LINEIF{$x=1$ and $y=1$}{\algorithmicreturn\ $Yes$.}
\LINEIF{$x=1$ or $y=1$}{\algorithmicreturn\ $No$.}
\STATE Using Algorithm \ref{al:PP} construct the support graph for $y$ in $\Sch_{d-1}(y)$ and the flow $\pi_y$.
\STATE Choose a prefix $y_i$ satisfying $\pi_y(y_i\rightarrow y_{i+1}) \ne 0$.
\FORALL{$x'$ such that $x=x'\circ x''$}\label{st:CP_for1}
    \STATE Put $\gamma = y_i x'^{-1}$.
    \STATE Using Algorithm \ref{al:PP} compute the flow $\pi_{\gamma x\gamma^{-1}}$ in $\Sch_{d-1}(y)$.
    \LINEIF{$\pi_{\gamma x\gamma^{-1}} \equiv \pi_{y}$}{\algorithmicreturn\ $Yes$.}
\ENDFOR\label{st:CP_for2}
\RETURN $No$.
\end{algorithmic}
\end{algorithm}

A few details are in order. To construct support graphs for $y$ and $\gamma x\gamma^{-1}$
in $\Sch_{d-1}(y)$ one needs to find all prefixes of $y$ and $\gamma x\gamma^{-1}$
define the same $\gp{y}$-cosets. The later problem reduces to the membership problem
for $\gp{y}$ and can be treated by Algorithm \ref{al:PP}. It follows from the choice of
$\gamma$'s that the inputs to Algorithm \ref{al:PP} have lengths bounded by $|x|+|y|$.

\begin{theorem}\label{th:deterministic_CP}
There exists a quasi-quintic time $\tO((|x|+|y|)^5)$ deterministic algorithm
solving the conjugacy problem in $S_{r,d}$.
\end{theorem}

\begin{proof}
Correctness of the algorithm follows from Theorem \ref{th:geometric_conjugacy}.
The loop \ref{st:CP_for1}--\ref{st:CP_for2} performs $|x|+1$ iterations.
At each iteration we compute the flow $\pi_{\gamma x\gamma^{-1}}$
which requires $O(|x|^2)$ runs of Algorithm \ref{al:PP}
and test if $\pi_{\gamma x\gamma^{-1}} \equiv \pi_{y}$ which requires
$O((|x|+|y|)^2)$ more runs of Algorithm \ref{al:PP}.
The deterministic Algorithm \ref{al:PP} has quasi-quadratic time complexity.
Thus, the total time complexity is $\tO((|x|+|y|)^5)$.
\end{proof}

We can further improve efficiency if we use the randomized version of Algorithm \ref{al:PP}.
Algorithm \ref{al:CP} invokes Algorithm \ref{al:PP} at most $(|x|+|y|)^2$ times on each iteration,
hence the total is number of invocations is bounded by $(|x|+|y|)^3$.
Each invocation of Algorithm \ref{al:PP} can produce an incorrect answer.
To better control the error we go deep into details of Algorithm \ref{al:PP} again.
As we mentioned above the lengths of inputs $(u,v)$ for Algorithm \ref{al:PP}
are bounded by $|x|+|y|$. Hence, for $T=T(\{u,v,[u,v]\})$ we have:
$$
\diam(T) \le |u|+|v|+|[u,v]| \le 5(|x|+|y|).
$$
Therefore, every time randomized Algorithm \ref{al:PP} is invoked it performs at most
$\log_3 (5(|x|+|y|)) \le 2+\log_3 (|x|+|y|)$ iterations.
The number of vertices in $T$ is also bounded by $5(|x|+|y|)$.
Hence, the total number of bad hyperplanes is not grater than $25(|x|+|y|)^2$.
Therefore, choosing a random tuple $A$ with elements in $[0,25(|x|+|y|)^6]$ produces
the correct result on a single iteration with probability not less than
$$
1-\tfrac{25(|x|+|y|)^2}{25(|x|+|y|)^6} \le 1-\tfrac{1}{(|x|+|y|)^4}.
$$
Hence, we get the correct result on a single invocation of an algorithm \ref{al:PP} with probability at least:
$$
\rb{1-\tfrac{1}{(|x|+|y|)^4}}^{2+\log_3 (|x|+|y|)}.
$$
Performing $(|x|+|y|)^3$ invocations of Algorithm \ref{al:PP} results in all correct results with probability
at least
$$
\rb{\rb{1-\tfrac{1}{(|x|+|y|)^4}}^{2+\log_3 (|x|+|y|)}}^{((|x|+|y|)^3)} \ge \rb{1-\tfrac{1}{|x|+|y|}}^{2+\log_3 (|x|+|y|)}.
$$
This proves the following theorem.

\begin{theorem}\label{th:randomized_CP}
There exists a quasi-quadric time $\tO((|x|+|y|)^4)$ unbiased randomized algorithm
solving the conjugacy problem in $S_{r,d}$. The probability of a correct computation
is at least $\rb{1-\tfrac{1}{|x|+|y|}}^{2+\log_3 (|x|+|y|)}$.
\qed
\end{theorem}

\providecommand{\bysame}{\leavevmode\hbox to3em{\hrulefill}\thinspace}
\providecommand{\MR}{\relax\ifhmode\unskip\space\fi MR }
\providecommand{\MRhref}[2]{%
  \href{http://www.ams.org/mathscinet-getitem?mr=#1}{#2}
}
\providecommand{\href}[2]{#2}


\begin{thebibliography}{10}

\bibitem{Agrawal-Kayal-Saxena:2004}
M.~{Agrawal}, N.~{Kayal}, and N.~{Saxena}, \emph{{PRIMES is in P}}, Ann. of
  Math. {160} (2004), pp.~781--–793.

\bibitem{Babai:1979}
L.~{Babai}, \emph{{Monte-Carlo algorithms in graph isomorphism testing}},
  Universit\'e de Montr\'eal, D.M.S. No. 79--10, 1979.

\bibitem{Fox_calc4}
K.~T. {Chen}, R.~H. {Fox}, and R.~C. {Lyndon}, \emph{Free differential calculus
  IV}, Ann. of Math. {71} (1960), pp.~408--422.

\bibitem{DLS}
C.~{Droms}, J.~{Lewin}, and H.~{Servatius}, \emph{{The length of elements in
  free solvable groups}}, Proc. Amer. Math. Soc. {119} (1993), pp.~27--33.

\bibitem{Fox_calc1}
R.~H. {Fox}, \emph{Free differential calculus I}, Ann. of Math. {57} (1953),
  pp.~547--560.

\bibitem{Fox_calc2}
\bysame, \emph{Free differential calculus II}, Ann. of Math. {59} (1954),
  pp.~196--210.

\bibitem{Fox_calc3}
\bysame, \emph{Free differential calculus III}, Ann. of Math. {64} (1956),
  pp.~407--419.

\bibitem{Kapovich_Miasnikov:2002}
I.~{Kapovich} and A.~G. {Miasnikov}, \emph{Stallings foldings and subgroups of
  free groups}, J. Algebra {248} (2002), pp.~608--668.

\bibitem{Kargapolov-Rem:1966}
M.~I. {Kargapolov} and V.~N. {Remeslennikov}, \emph{{The conjugacy problem for
  free solvable groups}}, Algebra i Logika Sem. {5} (1966), pp.~15--25.
  (Russian).

\bibitem{Lyndon-Schupp:2001}
R.~{Lyndon} and P.~{Schupp}, \emph{Combinatorial Group Theory}, Classics in
  Mathematics. Springer, 2001.

\bibitem{Magnus:1939}
W.~{Magnus}, \emph{{On a theorem of Marshall Hall}}, Ann. of Math. {40} (1939),
  pp.~764--768.

\bibitem{Magnus-Karrass-Solitar:1977}
W.~{Magnus}, A.~{Karrass}, and D.~{Solitar}, \emph{Combinatorial Group Theory}.
  Springer-Verlag, 1977.

\bibitem{Malcev:1960}
A.~{Malcev}, \emph{{On free solvable groups}}, Soviet Math. Doklady {1} (1960),
  pp.~65--68.

\bibitem{Matthews:1966}
J.~{Matthews}, \emph{{The conjugacy problem in wreath products and free
  metabelian groups}}, Trans. Amer. Math. Soc. {121} (1966), pp.~329--339.

\bibitem{Miasnikov_Romankov_Ushakov_Vershik:2010}
A.~G. {Miasnikov}, V.~{Romankov}, A.~{Ushakov}, and A.~{Vershik}, \emph{The
  word and geodesic problems in free solvable groups}, Trans. Amer. Math. Soc.
  {362} (2010), pp.~4655--4682.

\bibitem{Motwani-Raghavan:1995}
R.~{Motwani} and P.~{Raghavan}, \emph{{Randomized Algorithms}}. Cambridge
  University Press, 1995.

\bibitem{Naik-Regan-Sivakumar:1995}
A.~{Naik}, K.~{Regan}, and D.~{Sivakumar}, \emph{{On quasilinear-time
  complexity theory}}, Theoret. Comput. Sci. {148} (1995), pp.~325--349.

\bibitem{Rabin:1976}
M.~{Rabin}, \emph{{Probabilistic algorithms}}. Algorithms and complexity: new
  directions and recent results, pp.~21--39. Academic Press, 1976.

\bibitem{Rabin:1980}
\bysame, \emph{{Probabilistic tests for primality}}, J. of Number Theory {12}
  (1980), pp.~128--138.

\bibitem{Remeslennikov-Sokolov:1970}
V.~N. {Remeslennikov} and V.~G. {Sokolov}, \emph{{Certain properties of Magnus
  embedding}}, Algebra i Logika {9} (1970), pp.~566--578.

\bibitem{Schonhage-Strassen:1971}
A.~{Sch\"onhage} and V.~{Strassen}, \emph{{Schnelle multiplikation gro{\ss}er
  zahlen}}, Computing {7} (1971), pp.~281--292.

\bibitem{Solovay-Strassen:1977}
R.~{Solovay} and V.~{Strassen}, \emph{{A fast Monte-Carlo test for primality}},
  SIAM J. Comput. {6} (1977), pp.~84--85.

\bibitem{Stallings:1983}
J.~{Stallings}, \emph{{Topology of finite graphs}}, Invent. Math. {71} (1983),
  pp.~551--565.

\bibitem{Vassilieva:2011}
S.~{Vassilieva}, \emph{{Polynomial time conjugacy in wreath products and free
  solvable groups}}, Groups Complex. Cryptol. {3} (2011), pp.~105--120.

\bibitem{Vershik_Dobrynin:2004}
A.~M. {Vershik} and S.~{Dobrynin}, \emph{{Geometrical approach to the free
  sovable groups}}, Int. J. Algebra Comput. {15} (2005), pp.~1243--1260.

\end{thebibliography}
\end{document}